%% file: main.tex
\documentclass[11pt,reqno]{amsart}

\usepackage[T1]{fontenc}
\usepackage[a4paper,margin=0.9in]{geometry}
\usepackage{amsmath,amssymb,amsthm,mathtools}
\usepackage{enumitem}
\usepackage{microtype}
\usepackage[hidelinks]{hyperref}

\hypersetup{
  pdftitle={Endpoint Energy Atoms Force Local Pressure Concentration in Three-Dimensional Navier--Stokes Flow},
  pdfauthor={Hao Huang},
  pdfsubject={Endpoint kinetic-energy atoms and local pressure concentration},
  pdfkeywords={Navier--Stokes equations, endpoint energy atom, pressure concentration, Oseen evolution, Hardy space, BMO}
}

\allowdisplaybreaks
\numberwithin{equation}{section}

\newtheorem{theorem}{Theorem}[section]
\newtheorem{proposition}[theorem]{Proposition}
\newtheorem{lemma}[theorem]{Lemma}
\newtheorem{corollary}[theorem]{Corollary}
\newtheorem{definition}[theorem]{Definition}
\newtheorem{maintheorem}{Theorem}

\newtheorem{mainlocal}{Theorem}

\theoremstyle{remark}
\newtheorem{remark}[theorem]{Remark}

\newcommand{\T}{\mathbb T}
\newcommand{\R}{\mathbb R}
\newcommand{\Pp}{\mathbb P}
\newcommand{\Qq}{\mathbb Q}
\newcommand{\Hh}{\mathcal H}
\newcommand{\Aa}{\mathfrak A}
\newcommand{\BMO}{\mathrm{BMO}}
\newcommand{\dd}{\,\mathrm d}
\newcommand{\norm}[1]{\left\lVert#1\right\rVert}
\newcommand{\abs}[1]{\left|#1\right|}
\newcommand{\ip}[2]{\left\langle#1,#2\right\rangle}
\newcommand{\weakstarto}{\stackrel{*}{\rightharpoonup}}

\title[Endpoint Atoms and Pressure Concentration]{Endpoint Energy Atoms Force Local Pressure Concentration in Three-Dimensional Navier--Stokes Flow}
\author{Hao Huang}
\thanks{Corresponding author: huanghao92@skku.edu.}
\address{Department of Computer Science and Engineering, Sungkyunkwan University,
Suwon-si 16419, Gyeonggi-do, Republic of Korea}
\address{School of Mathematics and Statistics, Linyi University,
Linyi 276000, China}
\email{huanghao92@skku.edu}

\subjclass[2020]{Primary 35Q30, 35B44; Secondary 35K91, 42B30, 76D05}
\keywords{Navier--Stokes equations, endpoint energy atom, pressure concentration, Oseen evolution, Leray projection, Hardy--BMO duality}

\begin{document}

\begin{abstract}
We prove that a point atom in an endpoint kinetic-energy measure of a
three-dimensional incompressible Navier--Stokes flow on the flat torus forces
quantitative concentration of the actual pressure. At each restart time
\(\tau\), a same-state constraint--response comparison evolves \(u(\tau)\)
by the pressure-free componentwise advection--diffusion equation driven by
\(u\), and subtracts the resulting passive field \(z_\tau\).
Writing \(r_\tau=u-z_\tau\) and \(g_\tau=\Qq z_\tau\), we obtain the exact
relative pressure-work identity
\[
 \frac12\norm{r_\tau(t)}_2^2
 +\nu\int_\tau^t\norm{\nabla r_\tau}_2^2\dd\rho
 =\int_\tau^t\ip{\nabla p}{g_\tau}\dd\rho,
\]
where \(p\) is the Navier--Stokes pressure. If the atom has mass \(m\), Nash
smoothing makes \(z_\tau\) terminally non-atomic while \(r_\tau\) retains
terminal atomic mass at least \(m\); hence, after every fixed restart, the
terminal limsup of the accumulated pressure work is at least \(m/2\).
Mesoscopic localization yields
scale-explicit lower bounds for gauge-invariant \(L^2\) pressure oscillation
and for the local
\(L^2\) mass of the pressure gradient in shrinking terminal cylinders.
Consequently, on every terminal neighborhood of the atomic point, the
pressure modulo functions of time and its gradient fail to be square
integrable. The relative pressure-work principle extends to constrained
solenoidal Oseen evolutions: as
\(\tau\uparrow T_*\), the zero-initial responses vanish in weak-tail
topologies, although every fixed restart retains the same lower bound. This
identifies a necessary pressure-work mechanism, generated by the
incompressibility constraint, that links endpoint atoms to local pressure
concentration, without requiring a full regularity hypothesis.
\end{abstract}

\maketitle

\input{sections/introduction}
\input{sections/setting}
\input{sections/passive}
\input{sections/ledger}
\input{sections/atom}
\input{sections/negative}
\input{sections/localization}
\input{sections/application}
\input{sections/hardy}
\input{sections/sharpness}
\appendix
\input{sections/appendices}

\section*{Funding}
This research did not receive any specific grant from funding agencies in the
public, commercial, or not-for-profit sectors.

\section*{Declaration of competing interest}
The author declares no competing interests.

\section*{Data availability}
No data were used for the research described in this article.

\bibliographystyle{amsplain}
\bibliography{references}

\end{document}

%% file: sections/introduction.tex
\section{Introduction}
\label{sec:introduction}

Pressure is invisible in the global kinetic-energy balance of incompressible
flow. Yet we prove that if endpoint kinetic energy develops a point atom, the
actual Navier--Stokes pressure is forced to concentrate quantitatively in
every terminal neighborhood of that point. The result does not assert that
such an atom exists; it identifies a necessary local response of the equations
if one does. A point atom is the strongest loss of spatial compactness
compatible with the energy bound: it records a positive amount of kinetic
energy persisting at arbitrarily small scales, rather than merely a large
critical norm or a lower-dimensional singular set.

Let \(u\) be a Leray--Hopf solution of the unforced incompressible
Navier--Stokes equations on the flat three-dimensional torus, smooth before
a finite terminal time \(T_*\):
\begin{equation}
 \partial_tu+(u\cdot\nabla)u+\nabla p=\nu\Delta u,
 \qquad \nabla\cdot u=0.
 \label{eq:NS-intro}
\end{equation}
Although \(u\) is smooth for \(t<T_*\), its kinetic-energy densities may
lose strong compactness as \(t\uparrow T_*\). Weak-* limits of
\(\abs{u(t)}^2\dd x\) retain the concentration information lost by weak
\(L^2\) convergence. Energy measures and their concentration dimension were
introduced by Shvydkoy and developed systematically by Leslie and Shvydkoy
\cite{Shvydkoy2013,LeslieShvydkoy2018}; microlocal concentration was studied
by Arnold and Craig \cite{ArnoldCraig2010}. These viewpoints complement the
Leray theory of finite-energy solutions \cite{Leray1934}, partial regularity
\cite{Scheffer1976,CaffarelliKohnNirenberg1982}, endpoint regularity criteria
\cite{EscauriazaSereginSverak2003}, and quantitative concentration of
critical norms near possible singularities
\cite{BarkerPrange2020,BarkerPrange2021,MaekawaMiuraPrange2020}. Under a
Type-I condition, Chae and Wolf exclude atomic energy concentration for the
three-dimensional Euler equations \cite{ChaeWolf2020}. None of these results
identifies a necessary local mechanism in the Navier--Stokes pressure when
an endpoint kinetic-energy atom is present.

Pressure is the natural constraint force, but it is invisible in the global
kinetic-energy balance because gradients are orthogonal to divergence-free
velocity. Direct localization does not resolve the difficulty. At a late
restart, the velocity already contains the concentration accumulated before
that time, so an ordinary local energy balance cannot distinguish inherited
energy from the response generated after the restart. Existing local
pressure decompositions, pressure projections, and pressure-based
regularity criteria control pressure under additional assumptions
\cite{Wolf2017,ChamorroLemarieRieussetMayoufi2018,GuevaraPhuc2017}; they do
not derive a compulsory local pressure response from a hypothetical endpoint
atom.

We isolate that response by a controlled same-state comparison that cancels
the inherited state exactly. Suppose that an endpoint energy measure has an
atom of mass \(m>0\) at \(a\). At each
fixed restart time \(\tau\in(t_b,T_*)\), let \(z_\tau\) solve the
componentwise passive problem
\begin{equation}
 \partial_tz_\tau+(u\cdot\nabla)z_\tau=\nu\Delta z_\tau,
 \qquad z_\tau(\tau)=u(\tau).
 \label{eq:passive-intro}
\end{equation}
The Navier--Stokes velocity continues under the incompressibility constraint,
whereas \(z_\tau\) is transported and diffused componentwise under the same
drift and without pressure. Define
\begin{equation}
 g_\tau=\Qq z_\tau,
 \qquad r_\tau=u-z_\tau,
 \label{eq:NS-responses-intro}
\end{equation}
where \(\Qq\) is the potential Hodge projection. The two response fields
start from zero: \(g_\tau(\tau)=r_\tau(\tau)=0\). Because the two evolutions
have the same drift and the same restart state, the only changed ingredient is
the incompressibility constraint; their difference therefore isolates its
dynamical contribution. Passive smoothing makes
\(z_\tau\) terminally diffuse, so subtracting it removes the common initial
state and leaves terminal atomic mass at least \(m\) in the relative response.

The key identity is
\begin{equation}
 \frac12\norm{r_\tau(t)}_2^2
 +\nu\int_\tau^t\norm{\nabla r_\tau(\rho)}_2^2\dd\rho
 =\int_\tau^t\ip{\nabla p(\rho)}{g_\tau(\rho)}\dd\rho.
 \label{eq:identity-intro}
\end{equation}
The instantaneous integrand on the right need not have a sign, but its
accumulation from the common initial state is nonnegative because it equals
the two square terms on the left. Drift-independent Nash smoothing makes
\(z_\tau\) terminally non-atomic for every fixed restart, whereas
\(r_\tau\) retains terminal atomic mass at least \(m\). Hence
\begin{equation}
 \limsup_{t\uparrow T_*}
 \int_\tau^t\ip{\nabla p}{g_\tau}\dd\rho\ge\frac m2
 \label{eq:work-lower-intro}
\end{equation}
for every fixed \(\tau\in(t_b,T_*)\). Thus, after every fixed restart, the
terminal limsup of the accumulated pressure work is at least \(m/2\).

Our first main result, Theorem~\ref{thm:NS-pressure}, converts this global
pressure-work lower bound into local pressure concentration. Given any
\(R_n\downarrow0\), we first choose restart times \(\tau_n\uparrow T_*\) for
which the accumulated transport and diffusion lengths are negligible relative
to \(R_n\); with \(\tau_n\) fixed, atom inheritance then selects
\(t_n\uparrow T_*\). This order
\(R_n\longrightarrow\tau_n\longrightarrow t_n\) is part of the conclusion;
no limits are exchanged. A localized form of
\eqref{eq:identity-intro}, including the pressure flux created by the cutoff,
gives
\begin{equation}
 \liminf_{n\to\infty}
 \int_{\tau_n}^{t_n}\int
 \chi_{R_n}\nabla p\cdot g_{\tau_n}\dd x\dd t
 \ge\frac m2.
 \label{eq:local-work-intro}
\end{equation}
The mesoscopic restriction is intrinsic to this argument: when the cutoff
radius is comparable with the effective transport or diffusion length, the
boundary fluxes can be of the same order as the interior work, and the local
identity no longer separates the two.
Cauchy--Schwarz duality against
\(\nabla\cdot(\chi_{R_n}g_{\tau_n})\) yields a gauge-invariant
pressure-oscillation lower bound. The ball Poincar\'e inequality then converts
this oscillation bound into
\begin{equation}
 \liminf_{n\to\infty}
 R_n^2\int_{\tau_n}^{T_*}\int_{B_{2R_n}(a)}
 \abs{\nabla p}^2\dd x\dd t>0.
 \label{eq:pressure-gradient-intro}
\end{equation}
Thus, on every terminal neighborhood of \(a\), the pressure modulo functions
of time and its gradient fail to be square integrable. Equivalently, local
\(L^2\) control of either quantity excludes an endpoint point atom.

At temporal exponent \(q=2\), the atom-exclusion hypothesis is stronger than
the pressure integrability supplied by the energy class but weaker than
scale-critical pressure control. Indeed, \(u\in L_t^4L_x^3\) implies that the
zero-mean pressure belongs to \(L_t^2L_x^{3/2}\), whereas the criterion asks
for the local gain \(p-\beta(t)\in L_t^2L_x^2\). Under the Navier--Stokes
scaling \(p_\lambda(x,t)=\lambda^2p(\lambda x,\lambda^2t)\),
\(L_t^qL_x^r\) pressure control is critical when \(2/q+3/r=2\). Thus
\((q,r)=(2,2)\) remains supercritical because \(5/2>2\), while at \(r=2\)
the critical temporal exponent is \(q=4\). The result therefore separates
atomic energy concentration from the full singularity problem and supplies a
concrete intermediate target for pressure estimates, distinct from
pressure-oscillation criteria designed to prove regularity of the whole solution
\cite{GuevaraPhuc2017,Wolf2017,ChamorroLemarieRieussetMayoufi2018}.

The Navier--Stokes theorem is the specialization \(w=u\),
\(\nabla\pi_w=\nabla p\), of a general same-state constraint--response
principle. For a divergence-free drift with an integrable \(L^\infty\)
terminal tail and a constrained solenoidal advection--diffusion solution
carrying a terminal point atom, Theorem~\ref{thm:main} shows that, for each
fixed restart, \(r_\tau\) obeys the exact pressure-work identity,
\(r_\tau\) and \(c_\tau\) retain terminal atomic mass at least \(m\), and the
terminal work limsup is at least \(m/2\). At the same time, all three response
fields vanish in negative and time-integrated energy topologies as
\(\tau\uparrow T_*\).
Theorem~\ref{thm:localization} gives the mesoscopic form, concentrating this
work along ordered, radius-dependent restarts. These conclusions use only the
common-drift constrained/passive structure, not the Navier--Stokes
nonlinearity.

The work also admits a Hardy--BMO representation: each component of
\(w\cdot\nabla(\Qq z_\tau)\) is a periodic div--curl product, and the pairing
is invariant under spatially constant transport. A concentration example
shows why the separate weak-tail bounds alone do not force the signed pairing
to vanish. Nonautonomous Oseen and constrained drift--diffusion systems
provide the natural comparison framework
\cite{SilvestreVicol2012,HanselRhandi2014,AsamiHishida2025}.

A related preprint studies the same endpoint-atom setting through packet
genealogies, common backward adjoints, and saturation, obtaining same-parent
full-tail rigidity and divergence of delayed Oseen action
\cite{HuangFullTail2026}. Here we instead take \(w=u\) directly and establish
the actual-pressure square identity, its mesoscopic localization, and local
pressure non-integrability.

Section~\ref{sec:setting} states the principal Navier--Stokes theorem and the
global and mesoscopic forms of the relative pressure-work principle.
Section~\ref{sec:passive} proves passive
diffuseness. Sections~\ref{sec:ledger}--\ref{sec:negative} establish the
relative pressure-work identity, atom inheritance, and weak-tail estimates.
Section~\ref{sec:localization} proves the mesoscopic pressure bounds.
Section~\ref{sec:application} then proves the Navier--Stokes specialization.
Section~\ref{sec:hardy} gives the Hardy--BMO representation, and
Section~\ref{sec:sharpness} establishes the sharpness of the weak-tail
estimates and records the relation between restart times.

%% file: sections/setting.tex
\section{Setting and main results}
\label{sec:setting}

Let \(\Omega=\T^3=(\R/2\pi\mathbb Z)^3\), and fix \(\nu>0\) and
\(-\infty<t_b<s<T_*<\infty\). The Leray projector \(\Pp\) is the orthogonal
projection of periodic vector fields onto divergence-free fields. We take
\(\Pp\) to be the identity on the zero Fourier mode and put
\(\Qq=I-\Pp\), so \(\Qq\) annihilates constants. Both projectors commute
with spatial derivatives and the heat semigroup.

The prescribed drift \(u\) is real-valued, divergence-free, smooth on every
compact subinterval of \([t_b,T_*)\), and satisfies
\begin{equation}
 A(\sigma):=\int_\sigma^{T_*}\norm{u(t)}_\infty\dd t<\infty,
 \qquad \sigma\in[t_b,T_*).
 \label{eq:A-finite}
\end{equation}
For a smooth preterminal Navier--Stokes branch, this follows from the
Foias--Guillop\'e--Temam higher-derivative estimate and Agmon interpolation;
see Appendix~\ref{sec:NS-tail} and \cite{FoiasGuillopeTemam1981}.

Here \(L^2_\sigma(\Omega)\) denotes the closed subspace of divergence-free
vector fields in \(L^2(\Omega;\R^3)\). For
\(f\in L^2_\sigma(\Omega)\), let
\begin{equation}
 w(t)=U(t,s)f
 \label{eq:U-def}
\end{equation}
be the energy solution of
\begin{equation}
 \partial_tw+(u\cdot\nabla)w+\nabla\pi_w=\nu\Delta w,
 \qquad \nabla\cdot w=0,
 \qquad w(s)=f.
 \label{eq:Oseen}
\end{equation}
The pressure has zero spatial mean and obeys
\begin{equation}
 \Delta\pi_w=-\partial_i\partial_j(u_jw_i),
 \qquad
 \pi_w=R_iR_j(u_jw_i),
 \qquad
 \nabla\pi_w=-\Qq[(u\cdot\nabla)w].
 \label{eq:pressure}
\end{equation}
Here \(R_i=\partial_i(-\Delta)^{-1/2}\) on nonzero Fourier modes, and
repeated spatial indices are summed. For a zero-mean distribution \(q\), we
use
\(
 \norm{q}_{\dot H^{-1}}^2
 =\sum_{k\in\mathbb Z^3\setminus\{0\}}\abs{k}^{-2}\abs{\widehat q(k)}^2
\).

The componentwise passive propagator \(S(t,\tau)\) is defined by
\begin{equation}
 \partial_tz+(u\cdot\nabla)z=\nu\Delta z,
 \qquad z(\tau)=h.
 \label{eq:passive}
\end{equation}
No divergence constraint is imposed for \(t>\tau\).

\begin{definition}[Terminal atomic mass]
\label{def:atom}
For an \(L^2\)-bounded field \(q(t)\), a point \(a\in\Omega\), and a terminal
time \(T_*\), define
\begin{equation}
 \Aa_a[q]
 :=\lim_{\rho\downarrow0}\limsup_{t\uparrow T_*}
 \int_{B_\rho(a)}\abs{q(t,x)}^2\dd x.
 \label{eq:atomic-functional}
\end{equation}
We call \(q\) terminally non-atomic when \(\Aa_a[q]=0\) for every \(a\).
\end{definition}

For a scalar field \(q\), a time interval \(I\), and a ball \(B\subset
\Omega\), define its gauge-invariant space--time oscillation by
\begin{equation}
 \operatorname{Osc}_2(q;I\times B)
 :=\left[
 \int_I\inf_{\beta\in\R}
 \norm{q(t)-\beta}_{L^2(B)}^2\dd t
 \right]^{1/2},
 \label{eq:osc-def}
\end{equation}
with the value \(+\infty\) allowed.

\begin{maintheorem}[Navier--Stokes endpoint atoms force local pressure concentration]
\label{thm:NS-pressure}
Let \(u\) be a Leray--Hopf solution of \eqref{eq:NS-intro} that is smooth
on every compact subinterval of \([t_b,T_*)\). Suppose that for some sequence
\(t_k\uparrow T_*\),
\begin{equation}
 \abs{u(t_k,x)}^2\dd x\weakstarto\mu_*,
 \qquad \mu_*(\{a\})=m>0.
 \label{eq:NS-endpoint-atom}
\end{equation}
For \(\tau\in(t_b,T_*)\), let
\begin{equation}
 z_\tau=S(t,\tau)u(\tau),\qquad
 g_\tau=\Qq z_\tau,\qquad
 r_\tau=u-z_\tau.
 \label{eq:NS-restart-fields}
\end{equation}
Then, for every fixed \(\tau\in(t_b,T_*)\) and every \(t\in(\tau,T_*)\),
\begin{equation}
 \frac12\norm{r_\tau(t)}_2^2
 +\nu\int_\tau^t\norm{\nabla r_\tau(\rho)}_2^2\dd\rho
 =\int_\tau^t\ip{\nabla p(\rho)}{g_\tau(\rho)}\dd\rho,
 \label{eq:NS-main-ledger}
\end{equation}
and
\begin{equation}
 \limsup_{t\uparrow T_*}
 \int_\tau^t\ip{\nabla p}{g_\tau}\dd\rho\ge\frac m2.
 \label{eq:NS-work-lower}
\end{equation}
More quantitatively, fix \(s\in(t_b,T_*)\) and let \(R_n\downarrow0\).
There are \(\tau_n\uparrow T_*\), with \(\tau_n>s\), such that
\begin{equation}
 \frac{A(\tau_n)}{R_n}
 +\frac{\sqrt{\nu(T_*-\tau_n)}}{R_n}\longrightarrow0,
 \label{eq:NS-main-schedule}
\end{equation}
\begin{equation}
 \liminf_{n\to\infty}
 \operatorname{Osc}_2\!\left(
 p;(\tau_n,T_*)\times B_{2R_n}(a)\right)
 \ge c_\Omega\frac{m\sqrt\nu}{\norm{u(s)}_2},
 \label{eq:NS-main-osc-quant}
\end{equation}
and
\begin{equation}
 \liminf_{n\to\infty}
 R_n^2\int_{\tau_n}^{T_*}\int_{B_{2R_n}(a)}
 \abs{\nabla p}^2\dd x\dd t
 \ge c_\Omega\frac{m^2\nu}{\norm{u(s)}_2^2}.
 \label{eq:NS-main-grad-quant}
\end{equation}
In particular, for every \(t_0\in(t_b,T_*)\) and every sufficiently small
\(r>0\),
\begin{equation}
 \operatorname{Osc}_2
 \bigl(p;(t_0,T_*)\times B_r(a)\bigr)=+\infty
 \label{eq:NS-pressure-not-L2}
\end{equation}
and
\begin{equation}
 \int_{t_0}^{T_*}\int_{B_r(a)}\abs{\nabla p}^2\dd x\dd t=+\infty.
 \label{eq:NS-gradient-not-L2}
\end{equation}
Thus local square integrability of either the pressure modulo functions of
time or its gradient on one terminal neighborhood excludes a point atom of
the endpoint kinetic-energy measure at its center.
\end{maintheorem}

Theorem~\ref{thm:NS-pressure} is the Navier--Stokes specialization of a
same-drift constrained/passive principle. The next result isolates that
principle without using the nonlinear identity \(w=u\).

For a fixed Oseen solution \(w\) and each fixed restart time
\(\tau\in(s,T_*)\), set
\begin{equation}
 \begin{aligned}
 f_\tau&=w(\tau),& F_\tau&=\norm{f_\tau}_2,\\
 z_\tau(t)&=S(t,\tau)f_\tau,&
 v_\tau&=\Pp z_\tau,\\
 g_\tau&=\Qq z_\tau,&
 c_\tau&=w-v_\tau,\\
 r_\tau&=w-z_\tau=c_\tau-g_\tau.
 \end{aligned}
 \label{eq:restart-definitions}
\end{equation}
All three response fields \(r_\tau,c_\tau,g_\tau\) start from zero, while
\(v_\tau(\tau)=w(\tau)\). The Hodge relations are
\begin{equation}
 \Pp r_\tau=c_\tau,
 \qquad
 \Qq r_\tau=-g_\tau,
 \qquad
 c_\tau=r_\tau+g_\tau.
 \label{eq:Hodge-relations}
\end{equation}

\begin{maintheorem}[Relative pressure-work principle for a common drift]
\label{thm:main}
Let \(u,w\) satisfy \eqref{eq:A-finite}--\eqref{eq:Oseen}, and suppose
\begin{equation}
 \Aa_a[w]\ge m>0
 \label{eq:w-atom}
\end{equation}
for some \(a\in\Omega\). Then the following statements hold for every fixed
\(\tau\in(s,T_*)\).

\begin{enumerate}[label=\textup{(\roman*)},leftmargin=2.4em]
\item The passive field \(z_\tau\) and its Hodge components
\(v_\tau,g_\tau\) are
terminally non-atomic, whereas
\begin{equation}
 \Aa_a[r_\tau]\ge m,
 \qquad
 \Aa_a[c_\tau]\ge m.
 \label{eq:atom-inheritance-main}
\end{equation}

\item For every \(\tau<t<T_*\),
\begin{equation}
 \begin{aligned}
 \frac12\norm{r_\tau(t)}_2^2
 +\nu\int_\tau^t\norm{\nabla r_\tau(\rho)}_2^2\dd\rho
 &=\int_\tau^t\ip{\nabla\pi_w(\rho)}{g_\tau(\rho)}\dd\rho\\
 &=-\int_\tau^t\int_\Omega
 \pi_w\,\nabla\cdot z_\tau\dd x\dd\rho\\
 &=\int_\tau^t\int_\Omega
 u_iw_j\partial_j(g_\tau)_i\dd x\dd\rho.
 \end{aligned}
 \label{eq:main-ledger}
\end{equation}

\item The accumulated work and its product majorant satisfy
\begin{equation}
 \limsup_{t\uparrow T_*}
 \int_\tau^t\ip{\nabla\pi_w}{g_\tau}\dd\rho
 \ge\frac m2,
 \label{eq:main-work-lower}
\end{equation}
and
\begin{equation}
 \int_\tau^{T_*}
 \norm{u(t)}_\infty\norm{\nabla g_\tau(t)}_2\dd t
 \ge\frac{m}{2F_\tau}
 \ge\frac{m}{2\norm{f}_2}.
 \label{eq:main-product-lower}
\end{equation}
The first integral in \eqref{eq:main-product-lower} is understood in
\([0,+\infty]\).

\item For the tail function \(A\) in \eqref{eq:A-finite} and
\(q_\tau\in\{r_\tau,g_\tau,c_\tau\}\),
\begin{equation}
 \sup_{\tau<t<T_*}\norm{q_\tau(t)}_{\dot H^{-1}}
 \le C_\Omega F_\tau A(\tau)
 \label{eq:main-Hminus}
\end{equation}
and
\begin{equation}
 \norm{q_\tau}_{L^4((\tau,T_*);L^2)}
 \le C_\Omega\nu^{-1/4}F_\tau A(\tau)^{1/2}.
 \label{eq:main-L4}
\end{equation}
In particular, the quantities in \eqref{eq:main-Hminus} and
\eqref{eq:main-L4} tend to zero as \(\tau\uparrow T_*\).

\item The vector field
\begin{equation}
 \mathfrak h_{\tau,i}=w\cdot\nabla(g_\tau)_i
 \label{eq:Hardy-probe-main}
\end{equation}
belongs to the periodic real Hardy space for almost every time, has zero
mean componentwise, and obeys
\begin{equation}
 \norm{\mathfrak h_\tau(t)}_{\Hh^1}
 \le C_\Omega\norm{w(t)}_2\norm{\nabla g_\tau(t)}_2,
 \qquad
 \norm{\mathfrak h_\tau}_{L^2((\tau,T_*);\Hh^1)}
 \le C_\Omega\nu^{-1/2}F_\tau^2.
 \label{eq:Hardy-bounds-main}
\end{equation}
Moreover, the last line of \eqref{eq:main-ledger} is
\begin{equation}
 \int_\tau^t
 \ip{u}{\mathfrak h_\tau}_{\BMO,\Hh^1}\dd\rho.
 \label{eq:Hardy-pairing-main}
\end{equation}
\end{enumerate}
\end{maintheorem}

Theorem~\ref{thm:main} separates two quantifier regimes. As
\(\tau\uparrow T_*\), all three responses vanish in the stated weak tail
norms. For every fixed \(\tau\), however, \(r_\tau\) and \(c_\tau\) retain
atomic mass at least \(m\), and the terminal work limsup is at least \(m/2\).
This is the relative pressure-work mechanism underlying the result.

For the local theorem, fix a radial cutoff
\(\chi\in C_c^\infty(B_2(0))\) satisfying
\begin{equation}
 0\le\chi\le1,\qquad \chi=1\text{ on }B_1(0),
 \label{eq:cutoff-base}
\end{equation}
and, in geodesic coordinates centered at \(a\), put
\(\chi_R(x)=\chi((x-a)/R)\) for \(2R<\operatorname{inj}(\Omega)\). Define
\begin{equation}
 \delta_\tau=T_*-\tau.
 \label{eq:delta-def}
\end{equation}

\begin{mainlocal}[Mesoscopic form of the relative pressure-work principle]
\label{thm:localization}
Under the assumptions of Theorem~\ref{thm:main}, let
\(R_n\downarrow0\) be arbitrary. There exist
\(\tau_n\uparrow T_*\), with \(\tau_n>s\), and
\(t_n\in(\tau_n,T_*)\), \(t_n\uparrow T_*\), such that
\begin{equation}
 \frac{A(\tau_n)}{R_n}
 +\frac{\sqrt{\nu\delta_{\tau_n}}}{R_n}
 \longrightarrow0,
 \label{eq:mesoscopic-schedule}
\end{equation}
and
\begin{equation}
 \liminf_{n\to\infty}
 \int_{\tau_n}^{t_n}\int_\Omega
 \chi_{R_n}\nabla\pi_w\cdot g_{\tau_n}\dd x\dd t
 \ge\frac m2.
 \label{eq:localized-work-main}
\end{equation}
Consequently,
\begin{equation}
 \liminf_{n\to\infty}
 \int_{\tau_n}^{t_n}\int_{B_{2R_n}(a)}
 [\nabla\pi_w\cdot g_{\tau_n}]_+\dd x\dd t
 \ge\frac m2,
 \label{eq:positive-local-work-main}
\end{equation}
where \([h]_+=\max\{h,0\}\).
Moreover,
\begin{equation}
 \liminf_{n\to\infty}
 \operatorname{Osc}_2\!\left(
 \pi_w;(\tau_n,T_*)\times B_{2R_n}(a)\right)
 \ge c_{\Omega,\chi}\frac{m\sqrt\nu}{\norm{f}_2},
 \label{eq:pressure-osc-main}
\end{equation}
and
\begin{equation}
 \liminf_{n\to\infty}
 R_n^2\int_{\tau_n}^{T_*}\int_{B_{2R_n}(a)}
 \abs{\nabla\pi_w}^2\dd x\dd t
 \ge c_{\Omega,\chi}\frac{m^2\nu}{\norm{f}_2^2}.
 \label{eq:pressure-gradient-main}
\end{equation}
\end{mainlocal}

\begin{remark}[Order of limits]
\label{rem:order-limits}
The selection in Theorem~\ref{thm:localization} is ordered. First choose
\(R_n\), then choose a sufficiently late fixed restart \(\tau_n\) to make
the cutoff fluxes small, and only then use the terminal limsup of
\(r_{\tau_n}\) to choose \(t_n\). The theorem does not exchange these three
operations, nor does it localize at arbitrarily small radii for one fixed
restart.
\end{remark}

\begin{remark}[Scope]
\label{rem:scope-main}
Theorem~\ref{thm:NS-pressure} identifies a necessary consequence of an
endpoint atom; it does not assert that an atom exists. Its contrapositive is
an atom-exclusion criterion under local square integrability of
\(p-\beta(t)\) for some measurable function \(\beta\), or of \(\nabla p\),
not a proof that the unrestricted energy class has that regularity or a full
Navier--Stokes regularity criterion.
Theorem~\ref{thm:main} and its mesoscopic form are likewise conditional on
\(\Aa_a[w]>0\).
\end{remark}

%% file: sections/passive.tex
\section{Passive smoothing and terminal diffuseness}
\label{sec:passive}

We first isolate the only property of the passive propagator needed for
atom inheritance. Its smoothing constants do not depend on the size of the
divergence-free drift.

\begin{lemma}[Passive Nash smoothing]
\label{lem:Nash}
Let \(z=S(t,\tau)h\) solve \eqref{eq:passive}. For
\(1\le p\le\infty\),
\begin{equation}
 \norm{z(t)}_p\le\norm{h}_p
 \label{eq:Lp-contraction}
\end{equation}
whenever the right-hand side is finite. In addition,
\begin{equation}
 \norm{S(t,\tau)h}_\infty
 \le C_\Omega\bigl(1+[\nu(t-\tau)]^{-3/4}\bigr)\norm{h}_2,
 \qquad t>\tau.
 \label{eq:L2-Linfty}
\end{equation}
\end{lemma}

\begin{proof}
For finite \(p\), multiply the vector equation \eqref{eq:passive} by
\(\abs{z}^{p-2}z\), regularizing at zero when \(p<2\). The transport term
integrates to zero because \(\nabla\cdot u=0\), and diffusion is dissipative.
This proves \eqref{eq:Lp-contraction}; the endpoint cases follow by
approximation and the maximum principle.

For completeness, the periodic Nash inequality gives
\[
 \norm{\varphi}_2^{10/3}
 \le C_\Omega\norm{\varphi}_1^{4/3}
 \bigl(\norm{\nabla\varphi}_2^2+\norm{\varphi}_2^2\bigr).
\]
Combining it with the \(L^1\) contraction and the scalar energy identity
yields the standard \(L^1\to L^2\) bound
\[
 \norm{S(t,\tau)}_{L^1\to L^2}
 \le C_\Omega\bigl(1+[\nu(t-\tau)]^{-3/4}\bigr).
\]
The backward adjoint evolution has the same divergence-free drift structure
and the same bound. Duality gives \eqref{eq:L2-Linfty}. This is the
classical Nash argument \cite{Nash1958}; smooth approximation removes any
auxiliary regularity used in the calculation.
\end{proof}

\begin{proposition}[Terminal diffuseness]
\label{prop:diffuse}
For each fixed \(\tau\in(s,T_*)\), the fields \(z_\tau\), \(v_\tau=\Pp z_\tau\),
and \(g_\tau=\Qq z_\tau\) are terminally non-atomic. More precisely, for
each fixed \(\delta\in(0,T_*-\tau)\),
\begin{equation}
 \sup_{\tau+\delta<t<T_*}\norm{z_\tau(t)}_\infty<\infty,
 \label{eq:z-uniform}
\end{equation}
and for every \(2<p<\infty\),
\begin{equation}
 \sup_{\tau+\delta<t<T_*}
 \bigl(\norm{v_\tau(t)}_p+\norm{g_\tau(t)}_p\bigr)<\infty.
 \label{eq:PQ-uniform}
\end{equation}
\end{proposition}

\begin{proof}
Equation \eqref{eq:z-uniform} follows from Lemma~\ref{lem:Nash}. The
periodic Riesz transforms, hence \(\Pp\) and \(\Qq\), are bounded on \(L^p\)
for \(1<p<\infty\), which proves \eqref{eq:PQ-uniform}. Thus, for every
measurable set \(E\subset\Omega\),
\[
 \int_E\abs{z_\tau(t)}^2\dd x
 \le \norm{z_\tau(t)}_\infty^2\abs{E},
\]
whereas H\"older's inequality gives
\[
 \int_E\abs{v_\tau(t)}^2\dd x
 +\int_E\abs{g_\tau(t)}^2\dd x
 \le C_{p,\tau}\abs{E}^{1-2/p}
\]
on every late fixed subinterval. Taking \(E=B_\rho(a)\), then the terminal
limsup and finally \(\rho\downarrow0\), proves the claim.
\end{proof}

\begin{remark}
The conclusion is weaker than uniform small-set equiintegrability jointly in
\(\tau\) and \(t\), and that distinction is essential. The smoothing
constant in \eqref{eq:L2-Linfty} degenerates when \(t-\tau\downarrow0\).
\end{remark}

%% file: sections/ledger.tex
\section{Restart equations and the relative pressure-work identity}
\label{sec:ledger}

All identities below can first be proved on compact preterminal intervals,
where the fields are smooth, and then passed to the energy class by
approximation.

\begin{lemma}[Energy and Hodge relations]
\label{lem:energy-Hodge}
For \(\tau<t<T_*\),
\begin{equation}
 \begin{aligned}
 \norm{w(t)}_2^2+2\nu\int_\tau^t\norm{\nabla w}_2^2\dd\rho
 &=F_\tau^2,\\
 \norm{z_\tau(t)}_2^2+2\nu\int_\tau^t\norm{\nabla z_\tau}_2^2\dd\rho
 &=F_\tau^2.
 \end{aligned}
 \label{eq:two-energies}
\end{equation}
Moreover, \(\Pp\) and \(\Qq\) are orthogonal in \(L^2\) and in homogeneous
\(H^1\):
\begin{equation}
 \norm{q}_2^2=\norm{\Pp q}_2^2+\norm{\Qq q}_2^2,
 \qquad
 \norm{\nabla q}_2^2=\norm{\nabla\Pp q}_2^2+\norm{\nabla\Qq q}_2^2.
 \label{eq:Hodge-orthogonality}
\end{equation}
\end{lemma}

\begin{proof}
The transport terms in both equations are skew in \(L^2\), and the Oseen
pressure is orthogonal to the divergence-free field \(w\), which proves
\eqref{eq:two-energies}. Formula \eqref{eq:Hodge-orthogonality} follows
because \(\Pp,\Qq\) are complementary orthogonal Fourier multipliers and
commute with derivatives.
\end{proof}

\begin{lemma}[Restart system]
\label{lem:restart-system}
For every fixed \(\tau\in(s,T_*)\),
\begin{equation}
 r_\tau(\tau)=c_\tau(\tau)=g_\tau(\tau)=0
 \label{eq:zero-initial}
\end{equation}
and
\begin{equation}
 \partial_tr_\tau+(u\cdot\nabla)r_\tau+\nabla\pi_w
 =\nu\Delta r_\tau.
 \label{eq:r-equation}
\end{equation}
In addition,
\begin{equation}
 \partial_tc_\tau+\Pp[(u\cdot\nabla)c_\tau]
 =\nu\Delta c_\tau-\Pp[(\nabla u)^Tg_\tau],
 \label{eq:c-equation}
\end{equation}
and
\begin{equation}
 \partial_tg_\tau+\Qq[(u\cdot\nabla)g_\tau]
 =\nu\Delta g_\tau-\Qq[(u\cdot\nabla)v_\tau].
 \label{eq:g-equation}
\end{equation}
\end{lemma}

\begin{proof}
At the restart time, \(z_\tau(\tau)=w(\tau)\) is divergence-free, proving
\eqref{eq:zero-initial}. Subtracting \eqref{eq:passive} from
\eqref{eq:Oseen} gives \eqref{eq:r-equation}.

To obtain \eqref{eq:c-equation}, apply \(\Pp\) to the passive equation and
compare it with the Oseen equation. Since \(z_\tau=v_\tau+g_\tau\), use
\[
 (u\cdot\nabla)g_\tau
 =\nabla(u\cdot g_\tau)-(\nabla u)^Tg_\tau
\]
for the gradient field \(g_\tau\), and note that \(\Pp\) annihilates the
gradient term. Applying \(\Qq\) to the passive equation gives
\eqref{eq:g-equation}.
\end{proof}

\begin{theorem}[Oseen--passive relative pressure-work identity]
\label{thm:relative-ledger}
For every fixed \(\tau\in(s,T_*)\) and every \(t\in(\tau,T_*)\), the three
equalities in \eqref{eq:main-ledger} hold.
\end{theorem}

\begin{proof}
Take the \(L^2\) inner product of \eqref{eq:r-equation} with \(r_\tau\).
The divergence-free drift contributes zero, so
\begin{equation}
 \frac12\frac{\dd}{\dd t}\norm{r_\tau}_2^2
 +\nu\norm{\nabla r_\tau}_2^2
 =-\ip{\nabla\pi_w}{r_\tau}.
 \label{eq:r-energy-step}
\end{equation}
By \eqref{eq:Hodge-relations}, \(\Qq r_\tau=-g_\tau\). Because
\(\nabla\pi_w\) lies in the potential block,
\[
 -\ip{\nabla\pi_w}{r_\tau}
 =\ip{\nabla\pi_w}{g_\tau}.
\]
Integrating from \(\tau\), where \(r_\tau(\tau)=0\), proves the first
equality in \eqref{eq:main-ledger}. Integration by parts gives
\[
 \ip{\nabla\pi_w}{g_\tau}
 =-\int_\Omega\pi_w\nabla\cdot g_\tau\dd x
 =-\int_\Omega\pi_w\nabla\cdot z_\tau\dd x.
\]

Finally, \eqref{eq:pressure} and the fact that \(g_\tau\) is potential yield
\begin{align*}
 \ip{\nabla\pi_w}{g_\tau}
 &=-\ip{(u\cdot\nabla)w}{g_\tau}\\
 &=\int_\Omega u_jw_i\partial_j(g_\tau)_i\dd x\\
 &=\int_\Omega u_iw_j\partial_j(g_\tau)_i\dd x.
\end{align*}
The last step uses
\(\partial_j(g_\tau)_i=\partial_i(g_\tau)_j\) and then exchanges the dummy
indices.
\end{proof}

\begin{remark}[Accumulated positivity]
Theorem~\ref{thm:relative-ledger} does not say that
\(\ip{\nabla\pi_w(t)}{g_\tau(t)}\) is pointwise nonnegative. Its time
integral from the common zero initial state is nonnegative because it equals
relative energy plus viscous dissipation. The identity is not a third
independent energy budget; it is an exact recombination of the constrained
and passive dynamics through their Hodge blocks.
\end{remark}

%% file: sections/atom.tex
\section{Atom inheritance and the pressure-work lower bound}
\label{sec:atom}

\begin{proposition}[Preservation of terminal atomic mass]
\label{prop:inheritance}
If \(\Aa_a[w]\ge m>0\), then for every fixed \(\tau\in(s,T_*)\),
\[
 \Aa_a[r_\tau]\ge m,
 \qquad
 \Aa_a[c_\tau]\ge m.
\]
\end{proposition}

\begin{proof}
Fix \(\tau\) and choose \(\delta=(T_*-\tau)/2\).
Proposition~\ref{prop:diffuse} implies
\begin{equation}
 \lim_{\rho\downarrow0}
 \limsup_{t\uparrow T_*}\norm{z_\tau(t)}_{L^2(B_\rho(a))}=0.
 \label{eq:z-local-zero}
\end{equation}
The reverse triangle inequality and \(r_\tau=w-z_\tau\) give
\[
 \norm{r_\tau(t)}_{L^2(B_\rho(a))}
 \ge \norm{w(t)}_{L^2(B_\rho(a))}
      -\norm{z_\tau(t)}_{L^2(B_\rho(a))}.
\]
For fixed \(\rho\), first take the terminal limsup to obtain
\[
 \limsup_{t\uparrow T_*}\norm{r_\tau(t)}_{L^2(B_\rho(a))}
 \ge \max\!\left\{
 \limsup_{t\uparrow T_*}\norm{w(t)}_{L^2(B_\rho(a))}
 -\limsup_{t\uparrow T_*}\norm{z_\tau(t)}_{L^2(B_\rho(a))},0
 \right\}.
\]
By \eqref{eq:atomic-functional} and monotonicity and continuity of the square
root, every
\(L^2\)-bounded field \(q\) satisfies
\[
 \Aa_a[q]^{1/2}
 =\lim_{\rho\downarrow0}\limsup_{t\uparrow T_*}
   \norm{q(t)}_{L^2(B_\rho(a))}.
\]
Letting \(\rho\downarrow0\) and using \eqref{eq:z-local-zero} therefore gives
\(\Aa_a[r_\tau]^{1/2}\ge\Aa_a[w]^{1/2}\). Both sides are nonnegative, so
squaring at this stage yields \(\Aa_a[r_\tau]\ge\Aa_a[w]\). Repeating the
argument with
\(c_\tau=w-v_\tau\) and the \(L^p\)-diffuseness of \(v_\tau\) proves the
second inequality. No local Hodge orthogonality is used.
\end{proof}

\begin{proposition}[Every-restart work lower bound]
\label{prop:work-tax}
Under the assumptions of Proposition~\ref{prop:inheritance},
\eqref{eq:main-work-lower} and \eqref{eq:main-product-lower} hold for every
fixed \(\tau\in(s,T_*)\).
\end{proposition}

\begin{proof}
By Proposition~\ref{prop:inheritance}, there is a sequence
\(t_k\uparrow T_*\) such that
\begin{equation}
 \norm{r_\tau(t_k)}_2^2\ge m-o(1).
 \label{eq:r-global-lower}
\end{equation}
Insert this sequence into \eqref{eq:main-ledger} and retain the nonnegative
dissipation term. This gives \eqref{eq:main-work-lower}.

The last representation in \eqref{eq:main-ledger} satisfies
\begin{equation}
 \left|\int_\Omega u_iw_j\partial_j(g_\tau)_i\dd x\right|
 \le \norm{u}_\infty\norm{w}_2\norm{\nabla g_\tau}_2
 \le F_\tau\norm{u}_\infty\norm{\nabla g_\tau}_2,
 \label{eq:direct-majorant}
\end{equation}
where the Oseen energy identity gives \(\norm{w(t)}_2\le F_\tau\).
Since \(m\le\Aa_a[w]\le F_\tau^2\), we have \(F_\tau>0\).
Combining \eqref{eq:direct-majorant} with \eqref{eq:main-work-lower} proves
the first inequality in \eqref{eq:main-product-lower}. The second follows
from \(F_\tau\le\norm{f}_2\).
\end{proof}

\begin{remark}[A sufficient exclusion condition]
\label{rem:minimal-gate}
Define
\[
 I_f(\tau)=\int_\tau^{T_*}
 \norm{u(t)}_\infty\norm{\nabla g_\tau(t)}_2\dd t
 \in[0,+\infty].
\]
To exclude a terminal atom for one fixed solution, it would be enough to prove
\begin{equation}
 \liminf_{\tau\uparrow T_*}I_f(\tau)=0.
 \label{eq:minimal-gate}
\end{equation}
The full limit is unnecessary. Condition \eqref{eq:minimal-gate} is not
proved here; Proposition~\ref{prop:sharpness} below explains why it does not
follow from the currently separated tail norms.
\end{remark}

%% file: sections/negative.tex
\section{Weak-tail topology separation}
\label{sec:negative}

The response fields start from zero, which yields genuine smallness in a
negative spatial topology. On \(\T^3\), all of \(r_\tau,g_\tau,c_\tau\) have
zero spatial mean: \(w\) and \(z_\tau\) have the same conserved mean,
\(\Qq\) annihilates constants, and \eqref{eq:Hodge-relations} applies.
Thus homogeneous negative Sobolev norms are unambiguous.

\begin{proposition}[Negative-order tail smallness]
\label{prop:Hminus}
For \(q_\tau\in\{r_\tau,g_\tau,c_\tau\}\), estimate
\eqref{eq:main-Hminus} holds.
\end{proposition}

\begin{proof}
Applying \(\Qq\) to the passive equation and using
\(\Qq z_\tau(\tau)=0\), the heat Duhamel formula gives
\begin{equation}
 g_\tau(t)
 =-\int_\tau^t e^{\nu(t-\rho)\Delta}
 \Qq\nabla\cdot(z_\tau\otimes u)(\rho)\dd\rho.
 \label{eq:g-Duhamel}
\end{equation}
The heat semigroup and \(\Qq\) are bounded on \(\dot H^{-1}\), and
\begin{equation}
 \norm{\nabla\cdot F}_{\dot H^{-1}}\le\norm{F}_2.
 \label{eq:div-Hminus}
\end{equation}
Using \(\norm{z_\tau(\rho)}_2\le F_\tau\), we obtain
\begin{equation}
 \sup_{\tau<t<T_*}\norm{g_\tau(t)}_{\dot H^{-1}}
 \le C_\Omega F_\tau A(\tau).
 \label{eq:g-Hminus}
\end{equation}

The Duhamel formula for \eqref{eq:r-equation} starts from zero and contains
\(-\nabla\cdot(r_\tau\otimes u)-\nabla\pi_w\). The energy estimates give
\(\norm{r_\tau(t)}_2\le2F_\tau\), while the periodic Riesz transforms in
\eqref{eq:pressure} give
\begin{equation}
 \norm{\pi_w(t)}_2
 \le C_\Omega\norm{u(t)}_\infty\norm{w(t)}_2
 \le C_\Omega F_\tau\norm{u(t)}_\infty.
 \label{eq:pressure-L2}
\end{equation}
Together with
\(\norm{\nabla\pi_w}_{\dot H^{-1}}\le\norm{\pi_w}_2\), this yields
\begin{equation}
 \sup_{\tau<t<T_*}\norm{r_\tau(t)}_{\dot H^{-1}}
 \le C_\Omega F_\tau A(\tau).
 \label{eq:r-Hminus}
\end{equation}
Finally \(c_\tau=r_\tau+g_\tau\), proving the proposition.
\end{proof}

\begin{proposition}[Time-integrated energy smallness]
\label{prop:L4}
For \(q_\tau\in\{r_\tau,g_\tau,c_\tau\}\), estimate
\eqref{eq:main-L4} holds.
\end{proposition}

\begin{proof}
Lemma~\ref{lem:energy-Hodge} gives
\begin{equation}
 \int_\tau^{T_*}\norm{\nabla g_\tau}_2^2\dd t
 \le \frac{F_\tau^2}{2\nu}.
 \label{eq:g-gradient}
\end{equation}
Since \(r_\tau=w-z_\tau\), the two energy identities in
\eqref{eq:two-energies} imply
\begin{equation}
 \int_\tau^{T_*}\norm{\nabla r_\tau}_2^2\dd t
 \le \frac{2F_\tau^2}{\nu}.
 \label{eq:r-gradient}
\end{equation}
The relation \(c_\tau=r_\tau+g_\tau\) gives the same type of bound for
\(c_\tau\).

For every zero-mean periodic field \(q\), Fourier Cauchy--Schwarz yields
\begin{equation}
 \norm{q}_2^2\le\norm{q}_{\dot H^{-1}}\norm{\nabla q}_2.
 \label{eq:Hminus-H1-interpolation}
\end{equation}
Square \eqref{eq:Hminus-H1-interpolation}, integrate in time, and combine
Proposition~\ref{prop:Hminus} with
\eqref{eq:g-gradient}--\eqref{eq:r-gradient}. Taking fourth roots proves
\eqref{eq:main-L4}.
\end{proof}

\begin{remark}[Why homogeneous norms matter]
The clean factor \(\nu^{-1/4}\) in \eqref{eq:main-L4} uses the homogeneous
interpolation \eqref{eq:Hminus-H1-interpolation} and the zero mean of the
restart responses. An inhomogeneous \(H^{-1}\)--\(H^1\) argument produces
an additional low-frequency term of order \((T_*-\tau)^{1/4}\).
\end{remark}

\begin{remark}[Topology separation]
The conclusion
\[
 q_\tau\longrightarrow0
 \quad\text{in}\quad
 L_t^\infty\dot H_x^{-1}\cap L_t^4L_x^2
\]
does not contradict \(\Aa_a[q_\tau]\ge m\) for each fixed \(\tau\). The
atomic functional contains a terminal time limsup and can be supported on
successively shorter time spikes, which are invisible to the \(L_t^4\) norm.
\end{remark}

%% file: sections/localization.tex
\section{Mesoscopic localization of the pressure-work lower bound}
\label{sec:localization}

The global identity \eqref{eq:main-ledger} does not localize by simply
inserting a cutoff: Hodge orthogonality is global, and multiplication by a
cutoff does not commute with \(\Pp\) or \(\Qq\). We therefore begin from the
relative equation itself.

\begin{lemma}[Local relative-energy identity]
\label{lem:local-energy}
Let \(\psi\ge0\) be smooth on \([\tau,t]\times\Omega\). Then
\begin{equation}
 \begin{aligned}
 &\frac12\int_\Omega\psi(t)\abs{r_\tau(t)}^2\dd x
 +\nu\int_\tau^t\int_\Omega\psi\abs{\nabla r_\tau}^2\dd x\dd\rho\\
 &=
 \frac12\int_\tau^t\int_\Omega
 (\partial_\rho\psi+u\cdot\nabla\psi+\nu\Delta\psi)
 \abs{r_\tau}^2\dd x\dd\rho\\
&\quad+
 \int_\tau^t\int_\Omega
 \psi\,\nabla\pi_w\cdot g_\tau\dd x\dd\rho
 +\int_\tau^t\int_\Omega
 \pi_w c_\tau\cdot\nabla\psi\dd x\dd\rho .
 \end{aligned}
 \label{eq:local-energy}
\end{equation}
\end{lemma}

\begin{proof}
Test \eqref{eq:r-equation} with \(\psi r_\tau\). Integration by parts in
time, transport, and diffusion gives
\[
 \begin{aligned}
 &\frac12\int\psi(t)\abs{r_\tau(t)}^2
 +\nu\int_\tau^t\int\psi\abs{\nabla r_\tau}^2\\
 &=
 \frac12\int_\tau^t\int
 (\partial_\rho\psi+u\cdot\nabla\psi+\nu\Delta\psi)
 \abs{r_\tau}^2
 -\int_\tau^t\int\psi r_\tau\cdot\nabla\pi_w .
 \end{aligned}
\]
The initial term vanishes because \(r_\tau(\tau)=0\). Now use
\(r_\tau=c_\tau-g_\tau\) and \(\nabla\cdot c_\tau=0\):
\[
 -\int\psi r_\tau\cdot\nabla\pi_w
 =\int\psi\nabla\pi_w\cdot g_\tau
 +\int\pi_w c_\tau\cdot\nabla\psi .
\]
This proves \eqref{eq:local-energy}.
\end{proof}

For the fixed spatial cutoff \(\chi_R\) from \eqref{eq:cutoff-base}, write
\begin{equation}
 \begin{aligned}
 E_{\tau,R}(t)
 &=\frac12\int\chi_R\abs{r_\tau(t)}^2,\\
 D_{\tau,R}(t)
 &=\nu\int_\tau^t\int\chi_R\abs{\nabla r_\tau}^2,\\
 W_{\tau,R}(t)
 &=\int_\tau^t\int\chi_R\nabla\pi_w\cdot g_\tau,\\
 B_{\mathrm{tr}}(t)
 &=\frac12\int_\tau^t\int
 (u\cdot\nabla\chi_R)\abs{r_\tau}^2,\\
 B_\nu(t)
 &=\frac{\nu}{2}\int_\tau^t\int
 (\Delta\chi_R)\abs{r_\tau}^2,\\
 B_p(t)
 &=\int_\tau^t\int
 \pi_w c_\tau\cdot\nabla\chi_R .
 \end{aligned}
 \label{eq:local-terms}
\end{equation}
Lemma~\ref{lem:local-energy} becomes
\begin{equation}
 E_{\tau,R}(t)+D_{\tau,R}(t)
 =W_{\tau,R}(t)+B_{\mathrm{tr}}(t)+B_\nu(t)+B_p(t).
 \label{eq:local-ledger}
\end{equation}
The term \(B_p\) is the pressure-cutoff flux created by local Hodge
nonorthogonality; it cannot be omitted.

\begin{lemma}[Cutoff-flux bounds]
\label{lem:cutoff-flux}
Let \(\delta_\tau\) be as in \eqref{eq:delta-def}. For every radius satisfying
\(2R<\operatorname{inj}(\Omega)\), uniformly for \(s<\tau<t<T_*\),
\begin{equation}
 \abs{B_{\mathrm{tr}}(t)}
 +\abs{B_p(t)}
 +\abs{B_\nu(t)}
 \le C_{\Omega,\chi}F_\tau^2
 \left[
 \frac{A(\tau)}{R}
 +\frac{\nu\delta_\tau}{R^2}
 \right].
 \label{eq:direct-flux}
\end{equation}
\end{lemma}

\begin{proof}
The energy estimates imply
\begin{equation}
 \norm{r_\tau}_{L^\infty_tL^2_x}
 +\norm{c_\tau}_{L^\infty_tL^2_x}
 \le C F_\tau,
 \qquad
 \norm{\nabla r_\tau}_{L^2_{t,x}}
 +\norm{\nabla c_\tau}_{L^2_{t,x}}
 \le C\nu^{-1/2}F_\tau.
 \label{eq:local-energy-budgets}
\end{equation}
Using \(\abs{\nabla\chi_R}\le C_\chi R^{-1}\) and the energy bound in
\eqref{eq:local-energy-budgets}, we find
\[
 \abs{B_{\mathrm{tr}}(t)}
 \le C\frac{F_\tau^2}{R}A(\tau)
\]
For the pressure-cutoff flux, Cauchy--Schwarz in space,
\eqref{eq:local-energy-budgets}, and \eqref{eq:pressure-L2} give explicitly
\begin{align}
 \abs{B_p(t)}
 &\le \frac{C_\chi}{R}\int_\tau^t
       \norm{\pi_w(\rho)}_2\norm{c_\tau(\rho)}_2\dd\rho\notag\\
 &\le C_{\Omega,\chi}\frac{F_\tau^2}{R}
       \int_\tau^t\norm{u(\rho)}_\infty\dd\rho
 \le C_{\Omega,\chi}\frac{F_\tau^2}{R}A(\tau).
 \label{eq:pressure-cutoff-flux-bound}
\end{align}
Since \(\abs{\Delta\chi_R}\le C_\chi R^{-2}\),
\[
 \abs{B_\nu(t)}
 \le C\frac{\nu F_\tau^2\delta_\tau}{R^2}.
\]
Combining the three estimates proves \eqref{eq:direct-flux}.
\end{proof}

\begin{proposition}[One-restart local lower bound]
\label{prop:one-restart-local}
If \(\Aa_a[w]\ge m\), then for every fixed \(\tau\in(s,T_*)\) and every
sufficiently small fixed \(R>0\),
\begin{equation}
 \begin{aligned}
 \limsup_{t\uparrow T_*}W_{\tau,R}(t)
 \ge \frac m2
 -C_{\Omega,\chi}F_\tau^2
 \left[
 \frac{A(\tau)}{R}
 +\frac{\nu\delta_\tau}{R^2}
 \right].
 \end{aligned}
 \label{eq:one-restart-local}
\end{equation}
\end{proposition}

\begin{proof}
Proposition~\ref{prop:inheritance} and the fact that
\(\chi_R=1\) on \(B_R(a)\) give
\[
 \limsup_{t\uparrow T_*}\int\chi_R\abs{r_\tau(t)}^2\ge m.
\]
Solve \eqref{eq:local-ledger} for \(W_{\tau,R}\), discard the nonnegative
\(D_{\tau,R}\), and apply Lemma~\ref{lem:cutoff-flux}.
\end{proof}

\begin{proof}[Proof of Theorem~\ref{thm:localization}]
By assumption, \(A(\tau)\to0\); also
\(\sqrt{\nu\delta_\tau}\to0\) as \(\tau\uparrow T_*\).
Given \(R_n\downarrow0\), choose \(\varepsilon_n\downarrow0\). After \(R_n\)
has been fixed, discard finitely many terms so that
\(2R_n<\operatorname{inj}(\Omega)\), and choose
\(\tau_n\uparrow T_*\) so late that
\begin{equation}
 A(\tau_n)\le\varepsilon_nR_n,
 \qquad
 \sqrt{\nu\delta_{\tau_n}}\le\varepsilon_nR_n.
 \label{eq:tau-choice}
\end{equation}
This proves \eqref{eq:mesoscopic-schedule}. Keeping \(\tau_n\) fixed, use
\(\Aa_a[r_{\tau_n}]\ge m\) to select recursively \(t_n\uparrow T_*\), with
\(t_n>\max\{t_{n-1},\tau_n,T_*-1/n\}\) (where \(t_0=t_b\)), such that
\begin{equation}
 \int\chi_{R_n}\abs{r_{\tau_n}(t_n)}^2\dd x
 \ge m-\varepsilon_n.
 \label{eq:tn-choice}
\end{equation}
Equations \eqref{eq:local-ledger}, \eqref{eq:direct-flux},
\eqref{eq:tau-choice}, and \eqref{eq:tn-choice} give
\[
 W_{\tau_n,R_n}(t_n)
 \ge\frac{m-\varepsilon_n}{2}
 -C_{\Omega,\chi}\norm{f}_2^2(\varepsilon_n+\varepsilon_n^2).
\]
Taking the lower limit proves \eqref{eq:localized-work-main}. Since
\(0\le\chi_{R_n}\le1\) and its support lies in \(B_{2R_n}(a)\),
\eqref{eq:positive-local-work-main} follows.

For the pressure consequence, integrate by parts:
\begin{equation}
 W_{\tau,R}(t)
 =-\int_\tau^t\int
 \pi_w\,\nabla\cdot(\chi_Rg_\tau)\dd x\dd\rho.
 \label{eq:local-pressure-duality}
\end{equation}
For every time,
\(\int_\Omega\nabla\cdot(\chi_Rg_\tau)\dd x=0\), so an arbitrary measurable
function \(\beta(\rho)\) may be subtracted from \(\pi_w\). Moreover,
\begin{align}
 \norm{\nabla\cdot(\chi_Rg_\tau)}_{L^2((\tau,T_*)\times\Omega)}
 &\le C_\chi
 \left(
 \norm{\nabla g_\tau}_{L^2_{t,x}}
 +R^{-1}\norm{g_\tau}_{L^2_{t,x}}
 \right)\notag\\
 &\le C_{\Omega,\chi}F_\tau
 \left(\nu^{-1/2}+\frac{\delta_\tau^{1/2}}{R}\right).
 \label{eq:div-cutoff-g}
\end{align}
The pointwise minimizer in \eqref{eq:osc-def} is the spatial mean of
\(\pi_w(t)\) on the ball. If the oscillation in
\eqref{eq:pressure-osc-main} is infinite, the desired inequality is
automatic. Otherwise, subtract this mean in
\eqref{eq:local-pressure-duality}. For all sufficiently large \(n\),
\eqref{eq:localized-work-main}, Cauchy--Schwarz, and
\eqref{eq:div-cutoff-g} give
\[
 W_{\tau_n,R_n}(t_n)
 \le C_{\Omega,\chi}
 \operatorname{Osc}_2\!\left(
 \pi_w;(\tau_n,T_*)\times B_{2R_n}(a)\right)
 F_{\tau_n}
 \left(\nu^{-1/2}+\frac{\delta_{\tau_n}^{1/2}}{R_n}\right).
\]
The denominator here satisfies, again for all sufficiently large \(n\),
\[
 F_{\tau_n}\left(
 \nu^{-1/2}+\frac{\delta_{\tau_n}^{1/2}}{R_n}\right)
 =\nu^{-1/2}F_{\tau_n}
 \left(1+\frac{\sqrt{\nu\delta_{\tau_n}}}{R_n}\right)
 \le 2\nu^{-1/2}\norm{f}_2
\]
by \eqref{eq:mesoscopic-schedule} and \(F_{\tau_n}\le\norm{f}_2\).
Taking the lower limit and using \eqref{eq:localized-work-main} proves
\eqref{eq:pressure-osc-main}.

For each \(n\), if the gradient integral below is infinite, the corresponding
estimate is automatic. Otherwise, the ball Poincar\'e inequality gives
\[
 \inf_{\beta\in\R}
 \norm{\pi_w(t)-\beta}_{L^2(B_{2R_n}(a))}
 \le C_\Omega R_n
 \norm{\nabla\pi_w}_{L^2(B_{2R_n}(a))}.
\]
Square, integrate in time, and use \eqref{eq:pressure-osc-main} to prove
\eqref{eq:pressure-gradient-main}.
\end{proof}

\begin{remark}[The diffusion-scale boundary]
For one fixed restart, the error in \eqref{eq:one-restart-local} need not
remain small as \(R\downarrow0\). At
\(R\sim\sqrt{\nu(T_*-\tau)}\), the diffusion-cutoff flux can already be of
order \(F_\tau^2\); an analogous obstruction occurs when the effective drift
length is comparable to \(R\). The local identity cannot then decide whether
energy was produced by interior pressure work or transported through the
cutoff boundary. Theorem~\ref{thm:localization} is therefore genuinely
mesoscopic rather than an unrestricted fixed-restart localization.
\end{remark}

%% file: sections/application.tex
\section{The Navier--Stokes specialization}
\label{sec:application}

We now prove Theorem~\ref{thm:NS-pressure}. The key specialization is that the
Navier--Stokes velocity is itself a constrained Oseen solution driven by the
same velocity. Thus, for every fixed
\(s\in(t_b,T_*)\),
\begin{equation}
 w(t)=U(t,s)u(s)=u(t),
 \qquad \nabla\pi_w=\nabla p
 \label{eq:w-equals-u}
\end{equation}
up to the irrelevant time-dependent pressure gauge. Indeed, \(u\) solves
\eqref{eq:Oseen} with datum \(u(s)\), and the energy estimate gives uniqueness
for that drift equation. No auxiliary solution or external reduction theorem
is required.

The endpoint measure in \eqref{eq:NS-endpoint-atom} gives the hypothesis of
Theorem~\ref{thm:main}. Indeed, for every sufficiently small \(\rho>0\),
the Portmanteau theorem yields
\begin{equation}
 \limsup_{t\uparrow T_*}
 \int_{B_\rho(a)}\abs{u(t,x)}^2\dd x
 \ge
 \liminf_{k\to\infty}
 \int_{B_\rho(a)}\abs{u(t_k,x)}^2\dd x
 \ge \mu_*(B_\rho(a)).
 \label{eq:portmanteau-atom}
\end{equation}
Since \(\mu_*\) is finite and \(B_\rho(a)\downarrow\{a\}\), continuity from
above and then \(\rho\downarrow0\) give
\begin{equation}
 \Aa_a[u]\ge\mu_*(\{a\})=m.
 \label{eq:NS-atomic-functional}
\end{equation}
Given a fixed \(\tau\in(t_b,T_*)\), choose any \(s\in(t_b,\tau)\).
Equations \eqref{eq:NS-main-ledger} and \eqref{eq:NS-work-lower} now follow
from Theorem~\ref{thm:main} with \(f=u(s)\), \(w=u\), and
\(\nabla\pi_w=\nabla p\).
The same specialization also gives, for every fixed restart,
\begin{equation}
 \int_\tau^{T_*}\norm{u(t)}_\infty
 \norm{\nabla g_\tau(t)}_2\dd t
 \ge \frac{m}{2\norm{u(\tau)}_2},
 \label{eq:NS-product-lower}
\end{equation}
while
\begin{equation}
 r_\tau,g_\tau,c_\tau\longrightarrow0
 \quad\text{in}\quad
 L_t^\infty\dot H_x^{-1}\cap L_t^4L_x^2
 \quad(\tau\uparrow T_*).
 \label{eq:NS-weak-tail}
\end{equation}

Fix \(s\in(t_b,T_*)\). Theorem~\ref{thm:localization} supplies the
quantitative local form. Given any \(R_n\downarrow0\), there are ordered choices
\(\tau_n<t_n<T_*\), with \(\tau_n,t_n\uparrow T_*\), such that
\begin{equation}
 \frac{A(\tau_n)}{R_n}
 +\frac{\sqrt{\nu(T_*-\tau_n)}}{R_n}\longrightarrow0
 \label{eq:NS-mesoscopic-schedule}
\end{equation}
and
\begin{equation}
 \liminf_{n\to\infty}
 \int_{\tau_n}^{t_n}\int
 \chi_{R_n}\nabla p\cdot g_{\tau_n}\dd x\dd t
 \ge\frac m2.
 \label{eq:NS-local-work}
\end{equation}
Along the same cylinders,
\begin{equation}
 \liminf_{n\to\infty}
 \operatorname{Osc}_2\!\left(
 p;(\tau_n,T_*)\times B_{2R_n}(a)\right)
 \ge c_{\Omega,\chi}
 \frac{m\sqrt\nu}{\norm{u(s)}_2}
 \label{eq:NS-local-pressure-osc}
\end{equation}
and
\begin{equation}
 \liminf_{n\to\infty}
 R_n^2\int_{\tau_n}^{T_*}\int_{B_{2R_n}(a)}
 \abs{\nabla p}^2\dd x\dd t
 \ge c_{\Omega,\chi}
 \frac{m^2\nu}{\norm{u(s)}_2^2}.
 \label{eq:NS-local-pressure-gradient}
\end{equation}
Fixing the standard cutoff \(\chi\) once and absorbing it into the constant
gives the constant \(c_\Omega\) stated in Theorem~\ref{thm:NS-pressure}.

It remains to derive the qualitative non-integrability statements. The
measurable pointwise spatial mean realizes the infimum in
\eqref{eq:osc-def}; hence finiteness of \(\operatorname{Osc}_2\) is equivalent
to square integrability after subtracting this time-dependent mean. Suppose
first that for some \(t_0\in(t_b,T_*)\), some \(r>0\), and some measurable
\(\beta:(t_0,T_*)\to\R\),
\begin{equation}
 p-\beta(t)\in L^2((t_0,T_*)\times B_r(a)).
 \label{eq:local-pressure-L2-assumption}
\end{equation}
Apply the quantitative result with \(s=t_0\), and choose any sequence
\(R_n\downarrow0\) with \(2R_n<r\). By absolute continuity of the integral,
the \(L^2\) norm of \(p-\beta(t)\) on
\((\tau_n,T_*)\times B_{2R_n}(a)\) tends to zero. Since the oscillation is
the infimum over all time-dependent gauges, this contradicts
\eqref{eq:NS-local-pressure-osc}. Hence
\eqref{eq:NS-pressure-not-L2} holds.

If instead \(\nabla p\in L^2((t_0,T_*)\times B_r(a))\), then the integrals
over the shrinking cylinders tend to zero, whereas
\eqref{eq:NS-local-pressure-gradient} forces them to be at least a positive
constant times \(R_n^{-2}\). This contradiction proves
\eqref{eq:NS-gradient-not-L2} and completes the proof of
Theorem~\ref{thm:NS-pressure}.

%% file: sections/hardy.tex
\section{Hardy--BMO form of the pressure work}
\label{sec:hardy}

Because every Hardy probe below has zero spatial mean, we use the mean-zero
part of the periodic real Hardy space \(\Hh^1(\T^3)\), together with periodic
BMO and VMO modulo constants. The corresponding duality and preduality follow
from the theory on complete spaces of homogeneous type; see
\cite[Theorem~B and Sec.~4, Theorem~4.1, pp.~638--639]{CoifmanWeiss1977},
with periodic VMO understood as the BMO-norm closure of smooth periodic
functions modulo constants. We shall use the
periodic Coifman--Rochberg--Weiss commutator estimate
\begin{equation}
 \norm{[R_j,b]f}_2
 \le C_\Omega\norm{b}_{\BMO}\norm{f}_2,
 \label{eq:periodic-CRW}
\end{equation}
where \([R_j,b]f=R_j(bf)-bR_jf\),
which follows from the Calder\'on--Zygmund commutator theorem on spaces of
homogeneous type \cite{AndersonDamian2022}, since the periodic Riesz
transforms are Calder\'on--Zygmund operators on the flat torus; the original Euclidean theorem is
\cite{CoifmanRochbergWeiss1976}. The argument below is the commutator proof
of the periodic CLMS div--curl estimate; compare the Euclidean theorem in
\cite{CoifmanLionsMeyerSemmes1993}. The resulting estimate is also stated
directly for periodic flows by Lions
\cite[Sec.~3.2, pp.~92--93, Eq.~(3.25)]{Lions1996Fluid}.

\begin{proposition}[Hardy-space estimate for the pressure-work integrand]
\label{prop:Hardy}
For almost every \(t\in(\tau,T_*)\), define
\begin{equation}
 \mathfrak h_{\tau,i}(t,x)
 =w_j(t,x)\partial_j(g_\tau)_i(t,x).
 \label{eq:Hardy-probe}
\end{equation}
Then \(\mathfrak h_\tau(t)\in\Hh^1(\Omega;\R^3)\),
\eqref{eq:Hardy-bounds-main} holds, and
\begin{equation}
 \int_\Omega\mathfrak h_{\tau,i}(t,x)\dd x=0
 \qquad\text{for each }i.
 \label{eq:Hardy-zero-mean}
\end{equation}
\end{proposition}

\begin{proof}
Fix a component \(i\) and a time for which both factors lie in \(L^2\).
Put
\[
 h=(-\Delta)^{1/2}\bigl((g_\tau)_i-\langle(g_\tau)_i\rangle\bigr),
 \qquad
 \partial_j(g_\tau)_i=R_jh.
\]
For every smooth periodic \(b\), skew-adjointness of \(R_j\) and
\(\sum_jR_jw_j=(-\Delta)^{-1/2}\nabla\cdot w=0\) give
\begin{align*}
 \int_\Omega b\,w_j\partial_j(g_\tau)_i\dd x
 &=-\int_\Omega h\sum_j R_j(bw_j)\dd x\\
 &=-\int_\Omega h\sum_j[R_j,b]w_j\dd x.
\end{align*}
Therefore \eqref{eq:periodic-CRW} and
\(\norm{h}_2=\norm{\nabla(g_\tau)_i}_2\) imply
\[
 \left|\int_\Omega b\,w\cdot\nabla(g_\tau)_i\dd x\right|
 \le C_\Omega\norm{b}_{\BMO}\norm{w}_2
 \norm{\nabla(g_\tau)_i}_2.
\]
The product has zero spatial mean by periodic integration by parts. By the
definition of periodic VMO above, the estimate extends from smooth tests to
VMO. Because the product has mean zero, periodic VMO--Hardy
preduality---equivalently, the Hardy norm characterization by smooth BMO
tests modulo constants---then gives
\begin{equation}
 \norm{w\cdot\nabla(g_\tau)_i}_{\Hh^1}
 \le C_\Omega\norm{w}_2\norm{\nabla(g_\tau)_i}_2.
 \label{eq:CLMS-bound}
\end{equation}
Summing over \(i\) proves the pointwise bound in
\eqref{eq:Hardy-bounds-main}. The Oseen energy estimate and
\eqref{eq:g-gradient} give
\[
 \norm{\mathfrak h_\tau}_{L^2_t\Hh^1_x}
 \le C_\Omega
 \sup_{t\ge\tau}\norm{w(t)}_2
 \norm{\nabla g_\tau}_{L^2_tL^2_x}
 \le C_\Omega\nu^{-1/2}F_\tau^2.
\]
Finally, periodic integration by parts and \(\nabla\cdot w=0\) give
\[
 \int_\Omega w\cdot\nabla(g_\tau)_i\dd x=0,
\]
which proves \eqref{eq:Hardy-zero-mean}.
\end{proof}

\begin{corollary}[Oscillation-only work]
\label{cor:BMO-work}
The relative pressure-work identity can be written as
\begin{equation}
 \frac12\norm{r_\tau(t)}_2^2
 +\nu\int_\tau^t\norm{\nabla r_\tau}_2^2\dd\rho
 =\int_\tau^t
 \ip{u}{\mathfrak h_\tau}_{\BMO,\Hh^1}\dd\rho.
 \label{eq:BMO-ledger}
\end{equation}
If \(\Aa_a[w]\ge m\), then
\begin{equation}
 \limsup_{t\uparrow T_*}
 \int_\tau^t
 \ip{u}{\mathfrak h_\tau}_{\BMO,\Hh^1}\dd\rho
 \ge\frac m2.
 \label{eq:BMO-tax}
\end{equation}
The right-hand pairing is unchanged when its BMO representative is altered
by a spatially constant velocity.
\end{corollary}

\begin{proof}
The identity follows from the last representation in
Theorem~\ref{thm:relative-ledger} and Hardy--BMO duality. The lower bound is
Proposition~\ref{prop:work-tax}. If \(b(t)\in\R^3\) is spatially constant,
then \eqref{eq:Hardy-zero-mean} gives
\(\ip{b(t)}{\mathfrak h_\tau(t)}=0\). Thus the pairing sees only the BMO
class of the velocity modulo constants.
\end{proof}

\begin{remark}
The positive accumulated identities obtained from different restart times
are neither independent nor automatically summable; their precise overlap
is recorded in Proposition~\ref{prop:restart-relation}.
\end{remark}

%% file: sections/sharpness.tex
\section{Sharpness of the weak-tail estimates and restart interaction}
\label{sec:sharpness}

The Navier--Stokes drift tail has finite, hence vanishing, \(L_t^1\BMO_x\)
mass; see Appendix~\ref{sec:NS-tail}. Proposition~\ref{prop:Hardy} supplies
a restart-uniform \(L_t^2\Hh_x^1\) bound. These statements alone do not make
the pairing in \eqref{eq:BMO-tax} vanish.

\begin{proposition}[A div--curl concentration example]
\label{prop:sharpness}
There exist a divergence-free field
\(V\in C^\infty(\T^3;\R^3)\), a divergence-free field
\(Y\in C^\infty(\T^3;\R^3)\), a gradient field
\(G\in C^\infty(\T^3;\R^3)\), a scalar \(\eta\in L^1(0,1)\), and, for every
\(\tau\in(0,1)\), a smooth \(b_\tau\) supported in \((\tau,1)\) such that
\begin{equation}
 \norm{b_\tau}_2\le1,
 \qquad
 \left(\int_\tau^1 \eta(t)b_\tau(t)\dd t\right)
 \left(\int_{\T^3}V\cdot[(Y\cdot\nabla)G]\dd x\right)=1.
 \label{eq:sharp-pairing}
\end{equation}
Consequently, \(u(t,x)=\eta(t)V(x)\) has vanishing
\(L^1((\tau,1);\BMO)\) tails, the fields
\begin{equation}
 \mathfrak k_\tau(t,x)=b_\tau(t)(Y\cdot\nabla)G(x)
 \label{eq:k-sharp}
\end{equation}
are uniformly bounded in \(L^2((\tau,1);\Hh^1)\), and yet their
Hardy--BMO pairings equal one for every \(\tau\).
\end{proposition}

\begin{proof}
Set
\[
 \phi(x)=\sin x_1,\qquad
 G=\nabla\phi=(\cos x_1,0,0),\qquad
 Y=(\cos x_2,0,0).
\]
Then \(\nabla\cdot Y=0\), \(G\) is a gradient, and
\[
 K:=(Y\cdot\nabla)G=(-\sin x_1\cos x_2,0,0).
\]
Take
\[
 V=(-\sin x_1\cos x_2,\ \cos x_1\sin x_2,0).
\]
It is divergence free and
\begin{equation}
 D:=\int_{\T^3}V\cdot K\dd x
 =\int_{\T^3}\sin^2x_1\cos^2x_2\dd x>0.
 \label{eq:D-positive}
\end{equation}

Choose \(\alpha\in(1/2,1)\) and \(\eta(t)=(1-t)^{-\alpha}\). Then
\(\eta\in L^1(0,1)\), but \(\eta\notin L^2(\tau,1)\) for every
\(\tau\in(0,1)\).
Therefore the linear functional
\(\psi\mapsto\int_\tau^1\eta\psi\) is unbounded on the \(L^2\) unit ball of
\(C_c^\infty(\tau,1)\). Select \(\psi_\tau\) in that unit ball with
\(\int \eta\psi_\tau\ge D^{-1}\), and rescale it by a factor in \((0,1]\) to
obtain \(b_\tau\) satisfying
\(\int \eta b_\tau=D^{-1}\). This proves \eqref{eq:sharp-pairing}.

The CLMS estimate applied to the fixed spatial div--curl product \(K\)
gives
\[
 \norm{\mathfrak k_\tau}_{L_t^2\Hh_x^1}
 =\norm{b_\tau}_2\norm{K}_{\Hh^1}
 \le\norm{K}_{\Hh^1}.
\]
The remaining statements follow from \(\eta\in L^1\) and
\eqref{eq:sharp-pairing}.
\end{proof}

\begin{remark}
Proposition~\ref{prop:sharpness} is not a Navier--Stokes counterexample and
does not claim that the constructed \(\mathfrak k_\tau\) comes from a restart
equation. It proves the exact functional-analytic limitation needed here:
separated \(L_t^1\BMO_x\) and \(L_t^2\Hh_x^1\) estimates, even with the
correct spatial div--curl structure, cannot by themselves close
\eqref{eq:minimal-gate}.
\end{remark}

\begin{proposition}[Cross-restart identity]
\label{prop:restart-relation}
For \(s<\sigma<\tau<t<T_*\),
\begin{equation}
 r_\sigma(t)=r_\tau(t)+S(t,\tau)r_\sigma(\tau).
 \label{eq:cross-restart}
\end{equation}
In general, the two terms on the right are not orthogonal in \(L^2\).
\end{proposition}

\begin{proof}
The passive evolution property gives
\(S(t,\tau)z_\sigma(\tau)=z_\sigma(t)\). Since
\(r_\sigma(\tau)=w(\tau)-z_\sigma(\tau)\),
\begin{align*}
 r_\tau(t)+S(t,\tau)r_\sigma(\tau)
 &=w(t)-S(t,\tau)w(\tau)\\
 &\quad+S(t,\tau)w(\tau)-S(t,\tau)z_\sigma(\tau)\\
 &=w(t)-z_\sigma(t)=r_\sigma(t).
\end{align*}
Both summands can have nonzero solenoidal and potential components, and no
Hodge relation makes them orthogonal.
\end{proof}

Equation \eqref{eq:cross-restart} prevents one from summing the left sides of
\eqref{eq:main-ledger} over restart times as independent contributions. Any
multiscale argument must first account for the overlap term
\(S(t,\tau)r_\sigma(\tau)\).

Taken together, Theorem~\ref{thm:NS-pressure} and the relative pressure-work
principle convert endpoint atomic concentration into a local integrability
obstruction for the actual pressure. The same-state comparison removes the
inherited part of the concentration and isolates the response generated by the
incompressibility constraint. Mesoscopic localization then forces a fixed
amount of relative pressure work into shrinking neighborhoods of the atomic
point after an ordered choice of radii, restart times, and terminal times.
Thus local \(L^2\) control of either the pressure modulo time-dependent
constants or its gradient excludes endpoint point atoms, providing a
pressure-estimate route to atom exclusion that is narrower than full
regularity. Endpoint atomic energy therefore cannot remain confined to the
velocity field: the incompressibility constraint forces a non-integrable local
signature in the actual pressure.

%% file: sections/appendices.tex
\section{Navier--Stokes tail integrability}
\label{sec:NS-tail}

We record why assumption \eqref{eq:A-finite} is available for a smooth
three-dimensional Navier--Stokes branch on a finite preterminal interval.
Fix the regular time \(t_b\) from Theorem~\ref{thm:NS-pressure}. For each
\(T\in(t_b,T_*)\), apply the energy inequality and the periodic
Foias--Guillop\'e--Temam estimate at derivative order two to the mean-zero
Galilean transform of \(u\) on \([t_b,T]\). On this transform the Stokes
operator is \(-\Delta\) on mean-zero solenoidal fields, and spatial
derivative norms are unchanged. The energy inequality gives
\begin{equation}
 \int_{t_b}^{T}\norm{\nabla u(t)}_2^2\dd t
 \le E_b:=\frac{\norm{u(t_b)}_2^2}{2\nu},
 \label{eq:energy-tail}
\end{equation}
while the derivative-order \(r=2\) case of the periodic
Foias--Guillop\'e--Temam higher-derivative estimate
\cite{FoiasGuillopeTemam1981} gives
\begin{equation}
 \int_{t_b}^{T}\norm{\Delta u(t)}_2^{2/3}\dd t
 \le C_{\mathrm{FGT}},
 \label{eq:FGT}
\end{equation}
where
\(C_{\mathrm{FGT}}=C_{\mathrm{FGT}}(\nu,\Omega,\norm{u(t_b)}_2,
T_*-t_b)<\infty\) is independent of \(T\). Monotone convergence therefore
permits \(T\uparrow T_*\) in \eqref{eq:energy-tail} and \eqref{eq:FGT}.

On \(\T^3\), Agmon's inequality, with
\(C_{\mathrm A}=C_{\mathrm A}(\Omega)\), and conservation of the spatial mean
imply
\begin{equation}
 \norm{u(t)}_\infty
 \le \abs{\langle u(t_b)\rangle}
 +C_{\mathrm A}\norm{\nabla u(t)}_2^{1/2}
              \norm{\Delta u(t)}_2^{1/2}.
 \label{eq:Agmon}
\end{equation}
H\"older's inequality with exponents \(4\) and \(4/3\) gives, uniformly for
\(T<T_*\),
\begin{align}
 \int_{t_b}^{T}
 \norm{\nabla u}_2^{1/2}\norm{\Delta u}_2^{1/2}\dd t
 &\le
 \left(\int_{t_b}^{T}\norm{\nabla u}_2^2\dd t\right)^{1/4}
 \left(\int_{t_b}^{T}\norm{\Delta u}_2^{2/3}\dd t\right)^{3/4}
 \le E_b^{1/4}C_{\mathrm{FGT}}^{3/4}.
 \label{eq:Agmon-integrated}
\end{align}
Letting \(T\uparrow T_*\), we obtain the explicit bound
\[
 A(t_b)\le (T_*-t_b)\abs{\langle u(t_b)\rangle}
 +C_{\mathrm A}E_b^{1/4}C_{\mathrm{FGT}}^{3/4}<\infty.
\]
Thus \(u\in L^1((t_b,T_*);L^\infty)\), and absolute continuity of the
integral gives \(A(\tau)\to0\).

The same estimates have a critical oscillation form. Interpolation and the
Sobolev--BMO embedding \cite{Strichartz1980} yield
\begin{equation}
 \norm{u-\langle u\rangle}_{\BMO}
 \le C_\Omega\norm{u-\langle u\rangle}_{\dot H^{3/2}}
 \le C_\Omega\norm{\nabla u}_2^{1/2}\norm{\Delta u}_2^{1/2}.
 \label{eq:BMO-FGT}
\end{equation}
Consequently,
\begin{equation}
 \int_\tau^{T_*}\norm{u(t)}_{\BMO}\dd t\longrightarrow0
 \qquad(\tau\uparrow T_*),
 \label{eq:BMO-tail}
\end{equation}
where BMO is taken modulo spatial constants.

\section{The projected commutator identity}
\label{sec:commutator-appendix}

For comparison with the full relative response, the solenoidal block
\(c_\tau=w-\Pp z_\tau\) also has a positive accumulated identity. Taking the
\(L^2\) inner product of \eqref{eq:c-equation} with \(c_\tau\) and integrating
by parts gives
\begin{equation}
 \frac12\norm{c_\tau(t)}_2^2
 +\nu\int_\tau^t\norm{\nabla c_\tau}_2^2\dd\rho
 =\int_\tau^t\int_\Omega
 u_i(c_\tau)_j\partial_j(g_\tau)_i\dd x\dd\rho.
 \label{eq:c-ledger}
\end{equation}
Indeed,
\[
 -\int c_{\tau,i}(\partial_i u_j)(g_\tau)_j\dd x
 =\int u_j\partial_i(c_{\tau,i}(g_\tau)_j)\dd x
 =\int u_j(c_\tau)_i\partial_i(g_\tau)_j\dd x,
\]
Renaming the dummy indices gives \eqref{eq:c-ledger}. The full relative identity
\eqref{eq:main-ledger} is preferable: its advected factor is the fixed Oseen
solution \(w\), and the direct majorant loses no extra factor from
\(\norm{c_\tau}_2\le2F_\tau\).